\documentclass[english,a4paper,11pt]{article}
\usepackage{lmodern}

\usepackage{amssymb,amsmath,amsfonts,amsthm,epsfig,tikz,amscd} 
\usepackage{authblk}   
\usepackage{hyperref}
\usepackage{cite}

\usepackage[width=0.85\textwidth ]{caption}

\newcommand{\rd}{\textcolor{red}}

\usepackage{hyperref}
\newtheorem{theorem}{Theorem}[section]

\newtheorem{prop}[theorem]{Proposition}

\title{Dynamics of planar integrable Kepler billiards with a focused hyperbolic branch}

\author[1]{Daniel Jaud}
\affil[1]{\footnotesize Gymnasium Holzkirchen, Germany, Daniel.Jaud.PhD@gmail.com}

\author[2]{Lei Zhao}
\affil[2]{\footnotesize School of Mathematical Sciences, Dalian University of Technology, zhao1899@dlut.edu.cn}

\begin{document}
\maketitle
\flushbottom


\begin{abstract}
\noindent
We consider the integrable dynamics of a planar Kepler billiard in the plane bounded by a branch of a hyperbola focused at the Kepler center. As a sequel to our previous work \cite{JZ2}, we use the existence of the foci-caustic circle to identify an elliptic curve on which the dynamics is linearized and we analyze Cayley's condition on $n$-periodic orbits in this setting. By taking various limits also we discuss the dynamics of planar Kepler billiards with boundary being a focused parabolic arc,  as well as a straight line. In this last case, our approach offers a different way to deduce the elliptic curve as in \cite{Felder}.
\end{abstract}

\vspace*{0.4cm}

\vspace*{0.4cm}
{\textbf{MSC-Classification:} 14H70, 37C79, 37J99, 37N05}


\section{Introduction and setup}
The classical Poncelet theorem asserts that, given two conic sections in the plane, the existence of a polygon inscribed in one conic and circumscribed about the other implies the existence of an infinite family of such polygons. From a modern point of view, the underlying dynamics can be linearized on an elliptic curve, with the Poncelet construction corresponding to a translation by a fixed element of this curve. If this translation is of finite order, every orbit is periodic. Restricting the construction to the real part of the elliptic curve then yields the classical Poncelet porism, namely the existence of a one-parameter family of real polygons with the same periodicity property. 

The relation between the Poncelet porism and billiard dynamics has long been an important theme in the theory of integrable systems. Besides the classical billiard inside an ellipse, a variety of mechanical billiard systems remain integrable in the presence of a potential. Of particular interest are billiards with a Kepler/Coulomb potential. Several integrable examples of such systems are known, including the case of a straight reflection wall \cite{B, GJ} as well as billiards bounded by conic sections whose focus coincides with the center of the Kepler force. The dynamical investigation has been carried out for the straight line boundary case in \cite{Felder, GR}, while the dynamics of Kepler billiard inside a focused ellipse has been studied in \cite{JZ2}.

In the present work, we investigate the dynamics of an integrable Kepler billiard with a hyperbolic reflection boundary. In our setting, the reflection wall is one branch of a non-degenerate hyperbola, and the Kepler center lies on a focus.  As in \cite{JZ2}, we shall again trace the iteration of the second orbital foci, which all lie on a common circle that we call \emph{foci-circle}.


Parallel to the elliptic case, a key geometric observation is that the lines  joining consecutive second orbital foci are all tangent to a second fixed circle, the \emph{foci-caustic} circle. Its center and radius depend only on the parameters of the hyperbolic boundary and on the conserved quantities of the Kepler billiard. Hence the dynamics of the iterated second foci can again be reduced to a \emph{bicircular Poncelet situation}: consecutive points lie on a fixed outer circle, while the secants joining them are tangent to another fixed circle.

This reduction provides a direct link between the hyperbolic Kepler billiard and the classical theory of elliptic curves. Indeed, after complexification, the two circles defining the Poncelet configuration determine a pencil of quadrics and an associated algebraic curve of genus one. The billiard map on the sequence of second foci is thereby represented by a translation on the corresponding elliptic curve as discussed in the classical work by Griffiths and Harris \cite{GH2}. Periodic billiard trajectories are characterized by the condition that this translation has finite order, which recovers the classical condition of Cayley on periodicity \cite{DR2}.

While the elliptic reflection boundary leads to a particularly transparent global geometry \cite{JZ}, the hyperbolic case exhibits a richer collection of real configurations. Depending on the chosen branch of the hyperbola the relative position of the foci-circle and the foci-caustic circle may vary substantially. This affects the structure of the real part of the associated elliptic curve and, consequently, the dynamical interpretation of the corresponding translation. The main purpose of this work is to analyze this change of dynamics in details.

To carry out our analysis, we fix the following setup: 

We consider the motion of a point mass moving under the influence of an attractive Keplerian force with center  $F$ and reflected at a branch of a hyperbola with $F$ as a focus. By convention 
we take the other focus of the hyperbola $F'=(0,0)$ as the origin of the reference frame and put $F=(-2 \, c_{\mathbf{K}},0)$, in which $c_{\mathbf{K}}$ denotes the linear eccentricity.  We denote respectively by
$${\cal{B}}_F:\frac{(x+c_{\mathbf{K}})^2}{a_{\mathbf{K}}^2}-\frac{y^2}{c_{\mathbf{K}}^2-a_{\mathbf{K}}^2}=1,~~~x < - c_{\mathbf{K}}+a_{\mathbf{K}}$$
for the branch of the hyperbola enclosing $F$, and
$${\cal{B}}_{F'}:\frac{(x+c_{\mathbf{K}})^2}{a_{\mathbf{K}}^2}-\frac{y^2}{c_{\mathbf{K}}^2-a_{\mathbf{K}}^2}=1,~~~x > -c_{\mathbf{K}}+a_{\mathbf{K}}$$
for the branch enclosing $F'$. We have denoted the semi-major axis by $a_{\mathbf{K}}$, which satisfies $0<a_{\mathbf{K}}<c_{\mathbf{K}}$.

As in the elliptic case though, the dynamics on both sides of ${\cal{B}}_F$ or ${\cal{B}}_{F'}$ are equivalently up to reversing time. We may assume that the dynamics takes place on the side of the reflection wall not containing $F$ to again avoid the problem of collisions with the center of force associated to an attractive potential\footnote{Note that by Levi-Civita regularization technique, we may as well consider the other side of the reflection wall if we regularize the collisions.}.

By Kepler's law the orbital semi-major axis $a$ is a function of the total energy $E\sim -\frac{1}{2a}$
 of the particle. It is a conserved quantity of the billiard system. Therefore to track the iteration of Keplerian orbits, equivalently we track the iteration of the second orbital foci $F_{i}$ as in \cite{JZ, JZ2}. These second orbital foci $F_i$ all lie on the foci-circle centered at $F'$  with radius $R$, so it is given by the expression
$$S(F', R):x^2+y^2=R^2.$$
The circle ${\cal{C}}_0=S(F', R)$ is our foci-circle.

Analogously as in the elliptic case analyzed in \cite{JZ2}, along a Kepler billiard orbit, the lines containing consecutive second orbital foci $F_iF_{i+1}$ are all tangent to a circle that we call \emph{foci-caustic circle.} This again reduces the analysis to a rather standard, bicircular situation of the Poncelet porism. Based on this and the work of \cite{GH} we identify the elliptic curve on which the dynamics is linearized as well as analyzing how to apply Cayley's criteria on $n$-periodic orbits in this setting. 

In the limit $a_{\mathbf{K}} \to 0$, we recover a planar Kepler billiard system with a line as reflection wall in the plane. This is the integrable Boltzmann system, which has been introduced in \cite{B} and whose integrability is established in \cite{GJ}. Its dynamics has been investigated extensively in \cite{Felder}, in which an associated elliptic curve has been identified using the idea of \cite{GH}. We reproduce this analysis in a simple, now standard way. Apart from the limiting case corresponding to a straight line boundary there exit two other limits, which we consider both corresponding to parabolic arcs. The first limiting case corresponds to a parabolic arc with focus at $F$. The particle trajectories in this case remains Keplerian ellipses. Due to the limiting prozedure the foci-circle degenerates to a straight line. The second limiting case is given by a parabola with focus at $F'$. For this case the Keplerian ellipses degenerate to parabolic arcs associated to the motion of a particle under the influence of a constant attractive force.
 
We organize this article as follows: 

In Section \ref{sec:parameter_range_hyp} we derive the admissible range of parameters $(a,R)$ with given $(a_{\mathbf{K}}, c_{\mathbf{K}})$. Only with these parameters can the point particle perform actual reflections with the reflection walls respectively.

In Section \ref{sec:two_circles_hyp} we analyze the change of the foci-circle and foci-caustic circle when varying the parameters $(a,R)$. 


Lifting both quadratic equations to the projective plane, we will derive in the following section the associated elliptic curve, as describe in \cite{GH}. From this we derive Cayley's conditions which result in conditions on the parameters for $n-$periodic orbits.

In Sections \ref{sec:limits} we start the discussion of limiting cases by first considering the case of a parabolic boundary enclosing $F'$. We derive the associated foci-caustic circle as well a s the admissible range of parameters. Similar to the hyperbolic case we discuss the relative positions of the foci- and foci-caustic circle and derive conditions for low periodic orbits. Section \ref{sec:limits2} concerns the limiting case where one obtains a parabolic boundary with focus at $F$. In this limit the foci-circle degenerates to a straight line. An accordance to the section before the admissible range as well as conditions for low periodic orbits are derived. We close our discussion in Section \ref{sec:limits3} with the last limiting case corresponding to a straight line boundary. Thereby we reproduce and compare results already obtained in \cite{Felder}.


\section{Parameter range for admissible reflections}
\label{sec:parameter_range_hyp}

Not every Kepler orbit intersects with (one of) the hyperbolic boundary arc. Thus, one needs to determine conditions on the admissible values $(a,R)$ in order for physical reflections with the boundaries to occur. This is similar to the discussions in \cite{JZ2} for the elliptic boundary case. We derive the admissible parameter ranges for the two hyperbolic boundaries separately.

\subsection{Parameter range for the branch  ${\cal{B}}_F$ enclosing $F$}

When the reflection wall is given by ${\cal{B}}_F$, we need the apocenter of the elliptic orbit to lie on the opposite side of ${\cal{B}}_F$ from the center of force at $F$.

We distinguish two distinct cases $0\leq R\leq a_{\mathbf{K}}+c_{\mathbf{K}}$ and $R>a_{\mathbf{K}}+c_{\mathbf{K}}$.

\begin{figure}[htb]
\centering
\includegraphics[scale=0.6]{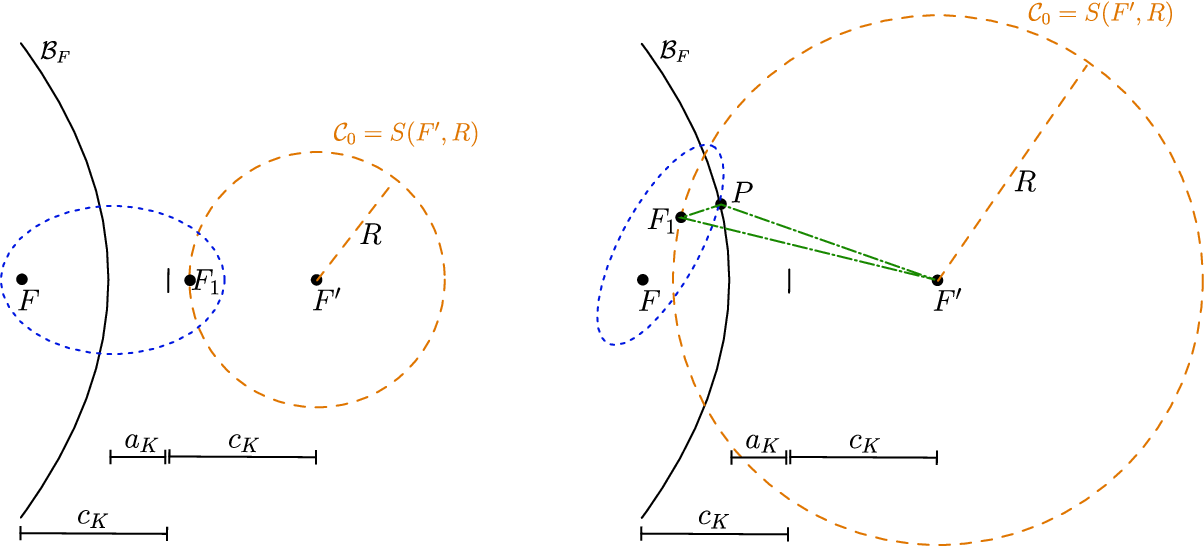}
\caption{Construction of admissible range.}
\label{fig:AdmissRange1}
\end{figure}

\underline{Case 1: $0\leq R \leq a_{\mathbf{K}}+c_{\mathbf{K}}$}

This means the foci-circle intersects $F F'$ at a point, say $F_{1}$. Further let $P$ be the intersection point of ${\cal{B}}_F$ with $FF'$. For this setup we find $|FF_1|>|FP|$. 

Then we need to have  $2 a \ge |FF_1|=2c_{\mathbf{K}}-R$ (see Fig. \ref{fig:AdmissRange1} left) and thus we get
$$a\geq c_{\mathbf{K}}-\frac{R}{2} ~~\Leftrightarrow ~~ R\geq 2(c_{\mathbf{K}}-a).$$

\underline{Case 2: $R > a_{\mathbf{K}}+c_{\mathbf{K}}$}

The parameters are admissible if and only if there is a point $F_{1}$ on $S(F', R)$ such that the associated Kepler ellipse with semi-major axis $a$ intersects ${\cal{B}}_F$ at a point $P$. Then we have

\begin{align*}
hyperbola:~|F'P|-|FP|&=2a_{\mathbf{K}},\\
ellipse:~|F_1P|+|FP|&=2a.
\end{align*}

Adding the both equations, we find
$$|F'P|+|F_1P|=2(a+a_{\mathbf{K}}).$$

Thus, $P$ lies on an ellipse with foci $F'$ and $F_1$ and with semi-major axis $a+a_{\mathbf{K}}$. By the triangle inequality, we find
$$|F'P|+|F_1P|\geq |F'F_1| ~~\Leftrightarrow ~~ 2(a+a_{\mathbf{K}})\geq R.$$

On the contrary, when $2(a+a_{\mathbf{K}})\geq R$, the triangle inequality is satisfied, and we may always find a degenerate Kepler ellipse realizing these parameters.

Combining these two different cases, we obtain the admissible parameter range stated in the following proposition: 
\begin{prop}
The admissible parameter range $(a,R)$ for a hyperbolic boundary arc enclosing $F$ is given by
\begin{align*}
R\leq 2(a+a_{\mathbf{K}}) ~\mbox{and}~ R\geq 2(c_{\mathbf{K}}-a)&~~\mbox{for} ~~\frac{c_{\mathbf{K}}-a_{\mathbf{K}}}{2}\leq a\leq c_\mathbf{K},\\
0\leq R\leq R\leq 2(a+a_{\mathbf{K}}) &~~\mbox{for}~~a>c_\mathbf{K}.
\end{align*}
\end{prop}
\label{prop:admiss_B_F}

See Fig. \ref{fig:AdmissRange1} for a graphical representation.

\begin{figure}[htb]
\centering
\includegraphics[scale=0.8]{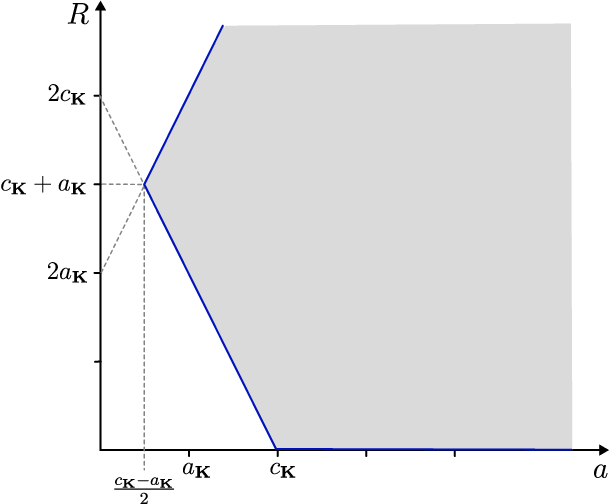}
\caption{Admissible region for hyperbolic arc enclosing $F$.}
\label{fig:AdmissRange1}
\end{figure}

\subsection{Hyperbolic boundary arc enclosing $F'$}
For the hyperbolic arc ${\cal{B}}_{F'}$ enclosing $F'$, the admissible range can be determined in a much simpler way. We consider the minimal degeneracy, where the Kepler orbit degenerates to a straight line segment and the second focus $F_i$ of this degenerate Kepler ``ellipse'' lies on the line $\overline{FF'}$. For $R<c_{\mathbf{K}}-a_{\mathbf{K}}$ the coordinates of $F_i$ are given by $F_i=(2_{\mathbf{K}}-R,0)$,  while we have set the origin $O=F'$. Such a degenerate Kepler orbit exists if and only if  $2a\geq 2c_{\mathbf{K}}-R$. Similarly, for $R\geq c_{\mathbf{K}}-a_{\mathbf{K}}$ the Kepler focus $F_i=(2c_{\mathbf{K}}-R,0)$ now lies to the left of the hyperbolic boundary arc ${\cal{B}}_{F'}$. For the degenerate Kepler rectilinear orbit to intersect ${\cal{B}}_{F'}$, we thus need to have
$$2a\geq c_{\mathbf{K}}+a_{\mathbf{K}}+a_{\mathbf{K}}-(c_{\mathbf{K}}-R)=2a_{\mathbf{K}}+R.$$ 
Combining these results, we find the restrictions on the admissible range:
\begin{align*}
R\geq 2c_{\mathbf{K}}-2a~~~\mbox{for}~~~R<c_{\mathbf{K}}-a_{\mathbf{K}},\\
R\leq 2a-2a_{\mathbf{K}}~~~\mbox{for}~~~R\geq c_{\mathbf{K}}-a_{\mathbf{K}}.
\end{align*}

For a graphical representation of the admissible values $(a,R)$ for this case refer to Fig. \ref{fig:AdmissRange2}.

\begin{figure}[htb]
\centering
\includegraphics[scale=0.8]{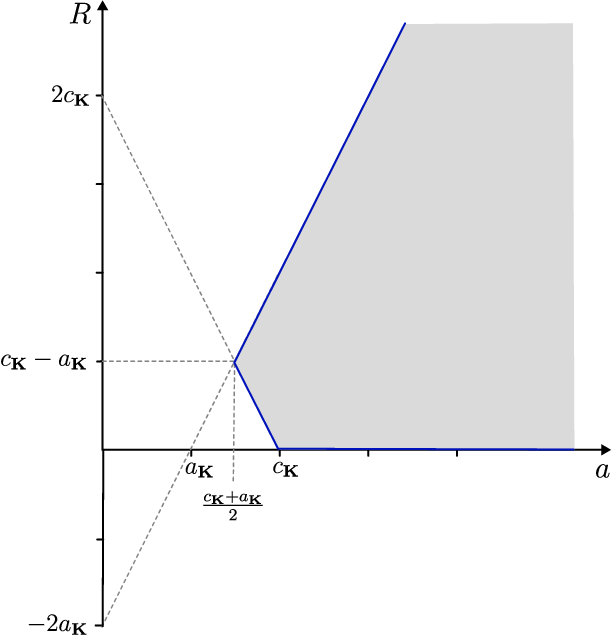}
\caption{Admissible region for hyperbolic arc enclosing $F'$.}
\label{fig:AdmissRange2}
\end{figure}

As expected, compared to the elliptic boundary cases as discussed in \cite{JZ2}, the hyperbolic admissible regions allow for a broader range of values. This is due to the fact, that in contrast to the ellipse, the hyperbolic branches are unbounded. 

\section{The foci-secants and the foci-caustic circle}
\label{sec:foci_caustic}

In \cite{JZ2}, we showed that, in the elliptic boundary case, the chords formed by consecutive second Kepler foci $F_i$ and $F_{i+1}$ are tangent to a common circle, which we called the \textit{foci-caustic circle}. The same is true for the hyperbolic boundary cases, with the minor geometric difference that the secants of consecutive second foci are now tangent to a common circle. This is summarized in the following theorem: 

\begin{theorem}
     In the hyperbolic boundary cases with ${\cal{B}}_{F'}$ or ${\cal{B}}_{F}$ as boundary, the secants formed by two consecutive second foci of the Kepler flight ellipses are tangent to the foci-caustic circle ${\cal{C}}_1$, given by
    $${\cal{C}}_1:(x-x_0)^2+y^2=r_0^2,$$
    where
    $$x_0=\frac{c_{\mathbf{K}}}{2(a_{\mathbf{K}}^2-c_{\mathbf{K}}^2)}\cdot [R^2-4(a\mp a_{\mathbf{K}})^2],$$
    and
    $$r_0=2(a\mp a_{\mathbf{K}})\mp \frac{a_{\mathbf{K}}}{c_{\mathbf{K}}}\cdot x_0.$$

    Here $\mp=-$ for  ${\cal{B}}_{F'}$ and $\mp=+$ for  ${\cal{B}}_{F}$.
    
\end{theorem}

For convenience of the reader, we include a direct proof.

\begin{proof}
    Recall that $F=(-2 c_{\mathbf{K}},0)$ and $F'=(0,0)$. 
    
    Write $b_{\mathbf{K}}=\sqrt{c_{\mathbf{K}}^2-a_{\mathbf{K}}^2}$. Let $P=(x_{P},y_{P})$ be a point of reflection along the reflection boundary, which is a branch of the hyperbola
    $${\cal{B}}: =\frac{(x_P+c_{\mathbf{K}})^2}{a_{\mathbf{K}}^2}-\frac{y_P^2}{b_{\mathbf{K}}^{2}}=1.$$
   We parametrize  ${\cal{B}}_{F}$ and ${\cal{B}}_{F'}$ respectively as 
    $$(x_P,y_P)=(\pm a_{\mathbf{K}}\cdot \cosh(t)-c_{\mathbf{K}},b_{\mathbf{K}}\cdot \sinh(t)),$$
    with $+$ for ${\cal{B}}_{F'}$ and $-$ for ${\cal{B}}_F$.
    Depending on whether $P$ lies on ${\cal{B}}_{F'}$ or ${\cal{B}}_{F}$, we have
    $$|FP|=|PF'| \pm 2 a_{\mathbf{K}}.$$

    The second foci $F_i$ and $F_{i+1}$ of consecutive Kepler flight ellipses are the intersection points of the foci-circle 
    $${\cal{C}}_1:x^2+y^2=R^2$$
    with the other circle 
    $${\cal{C}}_2:(x-x_P)^2+(y-y_P)^2=(2a-|P F|)^2=(2a \mp 2\cdot a_{\mathbf{K}}+|PF'|)^2$$
    obtained as the follows: At the point of reflection $P=(x_P,y_P)$ the two consecutive Kepler flight ellipses satisfy
   $$|FP|+|F_iP|=2a=|FP|+|F_{i+1}P| ~~\Leftrightarrow ~~ |F_iP|=|F_{i+1}P|.$$
   This means that $F_i$ and $F_{i+1}$ lie along the same circle ${\cal{C}}_2$ with midpoint $P$ and radius $r=|F_iP|=2a-|FP|$ as stated before.
    
    A straight forward calculation shows that 
    $$|PF'|=\sqrt{x_P^2+y_P^2}=c_{\mathbf{K}}\cosh(t) \pm a_{\mathbf{K}},$$
and thus 
     $${\cal{C}}_2:(x-x_P)^2+(y-y_P)^2=(2a \mp a_{\mathbf{K}}+c_{\mathbf{K}}\cosh(t))^2.$$

  The equation of the line $F_{1} F_{2}$ is obtained by subtracting the quadratic terms from the equations of ${\cal{C}}_1$ and ${\cal{C}}_2$. This is given by
\begin{align*}K(x, y, t)&:= 2 (\mp a_{\mathbf{K}} \cosh t-c_{\mathbf{K}}) x + 2  \sinh t \, b_{\mathbf{K}}\, y + 4 a^{2}\\
&~~~ + 4 (a \mp a_{\mathbf{K}}) c_{\mathbf{K}} \cosh t  \mp 4 a a_{\mathbf{K}}-R^{2}=0.
\end{align*}

The envelope curve of this family of lines, parametrized by $t$, is obtained by eliminating the parameter $t$ from the system of equations 
$$K(x,y,t)=0 ~~\hbox{  and }~~\partial_t K(x,y,t)=0.$$

    Solving the second equation we find:
    $$\tanh (t) = \dfrac{ b_{\mathbf{K}} y}{2 (a \mp a_{\mathbf{K}}) c_{\mathbf{K}} \mp a_{\mathbf{K}} x}.$$
    
Using this to eliminate $t$ from the above system of equations, we obtain the result through a lengthy but straightforward computation using identities for hyperbolic functions.

\end{proof}

For a graphical representation, see Figure \ref{fig:foci_caustic}.

\begin{figure}[htb]
\centering
\includegraphics[scale=0.8]{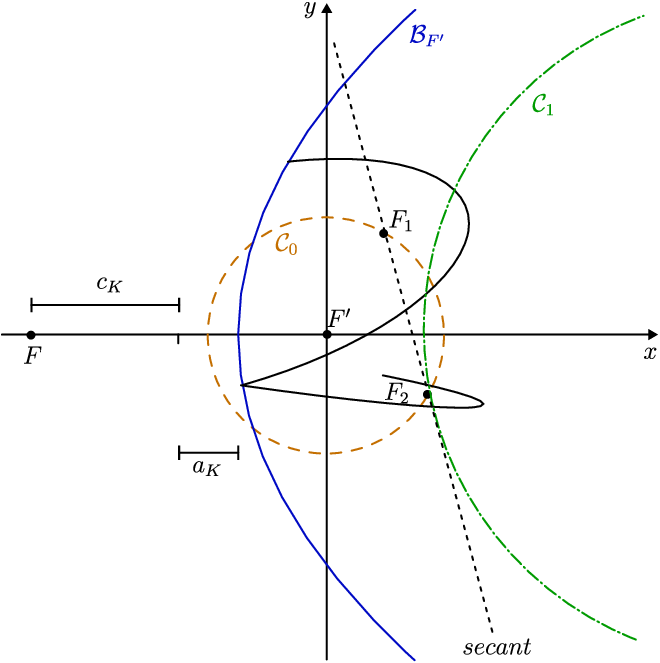}
\caption{Foci-circle and foci-caustic circle for ${\cal{B}}_{F'}$.}
\label{fig:foci_caustic}
\end{figure}

For comparison, the formula of the foci-caustic circle for the right hyperbolic arc ${\cal{B}}_{F'}$ is identical to the expression obtained for an elliptic boundary as derived in \cite{JZ2}. The only difference is the relative values of $a_{\mathbf{K}}$ and $c_{\mathbf{K}}$ (elliptic $a_{\mathbf{K}}>c_{\mathbf{K}}$; hyperbolic $a_{\mathbf{K}}<c_{\mathbf{K}}$).

\section{Relative positions of the foci- and foci-caustic circles}
\label{sec:two_circles_hyp}
Here we discuss the relative positions of the foci-circle ${\cal{C}}_0=S(F', R)$ as well as the foci-caustic circle ${\cal{C}}_1=S(x_{0}, |r_{0}|)$ for the case of reflections along a hyperbolic arc enclosing $F$, where we choose the origin of the coordinate system to lie at $F'$. These positions essentially determine the dynamics.

As seen in the last section, there is an associated parameter range for admissible reflections. We will fix an admissible $a$ and vary $R$ in the admissible range to analyze different scenarios for the two circles. 

We discuss the cases of ${\cal{B}}_{F}$ and ${\cal{B}}_{F'}$ separately. We illustrate the calculations for ${\cal{B}}_F$ in detail, whereas for ${\cal{B}}_{F'}$ we simply state the results, since the calculations are largely identical to those in the former case.

\subsection{The case of ${\cal{B}}_{F}$: Hyperbolic arc enclosing $F$}
\label{subsec:BF circles}
When considering the relative position of the two circles ${\cal{C}}_0=S(F',R)$ and ${\cal{C}}_1=S(x_0,|r_0|)$, where $F'=O$, two different initial configurations can occur in general: either the two circles are disjoint (Case 1), or one circle lies inside the other (Case 2). 

We start by discussing the first case. Define
$$d_1:=x_0-r_0-R~~~\mbox{for}~~~r_0\leq x_0,$$
which measures the distance between the two circles. From the restriction $r_0\leq x_0$ it follows $a\geq c_K$.
We can distinguish the following three following cases: 

(I) if $d_1>0$, there is no intersection of ${\cal{C}}_0$ and ${\cal{C}}_1$;

 (II) if $d_1=0$, the two circles ${\cal{C}}_0$ and ${\cal{C}}_1$ touch in one point;
 
 (III) if $d_1<0$, the two circles ${\cal{C}}_0$ and ${\cal{C}}_1$ intersect in two points, or become identical.

We study the three cases with $a\geq c_K$  in details. First, by substituting the corresponding expressions, we find:
$$d_1(R)=-\frac{1}{2(a_{\mathbf{K}}+c_{\mathbf{K}})}\cdot R^2+\frac{4(a+a_{\mathbf{K}})^2}{2(a_{\mathbf{K}}+c_{\mathbf{K}})}-2(a+a_{\mathbf{K}})-R,$$
which is a quadratic expression in $R$, from which on can easily determine the zeros of $d_1$:
$$d_1=0~~\Leftrightarrow ~~ R_\pm =\frac{-2(a_{\mathbf{K}}+c_{\mathbf{K}})\pm \sqrt{(2c_{\mathbf{K}}-2a_{\mathbf{K}}-4a)^2}}{2}.$$
With the restriction $a\geq c_{\mathbf{K}}>a_{\mathbf{K}}$ one obtains:
$$R_\pm =\begin{cases} 2(a-c_{\mathbf{K}})>0~~\mbox{for}~``+''\\
-2(a+a_{\mathbf{K}})<0~~\mbox{for}~``-''.
\end{cases}
$$
Note that only the positive root is of our interest for further considerations.

From the admissible parameter range, see Proposition \ref{prop:admiss_B_F}, thus we 
find:
\begin{enumerate}
\item for $0\leq R<2(a-c_{\mathbf{K}})$ there is no intersection between ${\cal{C}}_0$ and ${\cal{C}}_1$\rd{;}
\item for $R=2(a-c_{\mathbf{K}})$ the two circles ${\cal{C}}_0$ and ${\cal{C}}_1$ are tangent at one point,
\item for $2(a-c_{\mathbf{K}})< R<2(a+a_{\mathbf{K}})$ the circles ${\cal{C}}_0$ and ${\cal{C}}_1$ intersect in two points,
\item for $R=2(a+a_{\mathbf{K}})$ the two circles ${\cal{C}}_0$ and ${\cal{C}}_1$ are identical.
\end{enumerate}

Considering the second possible starting case, where ${\cal{C}}_0$ lies within ${\cal{C}}_1$. We define
$$d_2:=r_0-x_0-R~~~\mbox{for}~~~r_0>x_0,$$
which measures the minimal distance of these two circles. 

From $r_0>x_0$ we get
$$\frac{c_{\mathbf{K}}-a_{\mathbf{K}}}{2}\leq a<c_{\mathbf{K}}.$$
Similar to the previous calculation, we consider $d_2(R)$ as a function of $R$ and determine its zeros. We get
$$R_\pm =\begin{cases} 2(a+a_{\mathbf{K}})>0~~\mbox{for}~+\\
2(c_{\mathbf{K}}-a)<0~~\mbox{for}~- \rd{.}
\end{cases}
$$

With the associated admissible parameter range, in particular $\frac{c_{\mathbf{K}}-a_{\mathbf{K}}}{2}\leq a<c_{\mathbf{K}}$,  we now find only three different cases:
\begin{enumerate}
\item for $R=2(c_{\mathbf{K}}-a)$ the two circles ${\cal{C}}_0$ and ${\cal{C}}_1$ are tangent at one point with ${\cal{C}}_0$ inside ${\cal{C}}_1$,
\item for $2(c_{\mathbf{K}}-a)< R<2(a+a_{\mathbf{K}})$ the circles ${\cal{C}}_0$ and ${\cal{C}}_1$ intersect in two points,
\item for $R=2(a+a_{\mathbf{K}})$ the two circles ${\cal{C}}_0$ and ${\cal{C}}_1$ are identical. 
\end{enumerate}







 



\subsection{Hyperbolic arc enclosing $F'$}
For the setup of the hyperbolic boundary arc enclosing $F'$, exactly the same considerations can be used to analyze the relative positions of the two circles ${\cal{C}}_0$ and ${\cal{C}}_1$. We simply state our results.

For the admissible range $a\geq c_{\mathbf{K}}$, we find:
\begin{enumerate}
\item for $0\leq R<2(a-c_{\mathbf{K}})$ there is no intersection of ${\cal{C}}_0$ and ${\cal{C}}_1$ and the two circles lie outside of each other,
\item for $R=2(a-c_{\mathbf{K}})$ the two circles ${\cal{C}}_0$ and ${\cal{C}}_1$ are tangent at one point,
\item for $2(a-c_{\mathbf{K}})< R<2(a-a_{\mathbf{K}})$ the circles ${\cal{C}}_0$ and ${\cal{C}}_1$ intersect in two points,
\item for $R=2(a-a_{\mathbf{K}})$ the two circles ${\cal{C}}_0$ and ${\cal{C}}_1$ are identical.
\end{enumerate}

Similar, for the part of the admissible range where $\frac{c_{\mathbf{K}}+a_{\mathbf{K}}}{2}\leq a<c_{\mathbf{K}}$, we find:
\begin{enumerate}
\item for $R=2(c_{\mathbf{K}}-a)$ the two circles ${\cal{C}}_0$ and ${\cal{C}}_1$ are tangent at one point with ${\cal{C}}_0$ inside ${\cal{C}}_1$,
\item for $2(c_{\mathbf{K}}-a)< R<2(a-a_{\mathbf{K}})$ the circles ${\cal{C}}_0$ and ${\cal{C}}_1$ intersect in two points,
\item for $R=2(a-a_{\mathbf{K}})$ the two circles ${\cal{C}}_0$ and ${\cal{C}}_1$ are identical.
\end{enumerate}








 



\subsection{Focal reflection property}
\label{subsec:focal_reflection_property}

Considering the relative positions of the foci-circle and foci-caustic circle from the last section, there is one special case of interest, namely when the two circles meet at a single point. This occurs when $R=2(a-c_{\mathbf{K}})$ for ${\cal{B}}_F$ and ${\cal{B}}_{F'}$, respectively. This case is special because it leads to a useful geometric property, the focal reflection property, as obtained previously in related systems \cite{GR, JZ}. 

\begin{theorem}
For $R=2(a-c_{\mathbf{K}})$ all consecutive Kepler ellipses pass through the second focus $F'$. Furthermore the trajectories asymptotically approach the x-axis $\{y=0\}$. 
\end{theorem}

\begin{proof}
We know that all second orbital foci $F_i$ lie on the foci-circle ${\cal{C}}_0$, which has radius $R$. Any elliptic Kepler orbit passing through $F'$ satisfies
$$|FF'|+|F'F_i|=2a~~\Leftrightarrow ~~ 2c_{\mathbf{K}}+R=2a.$$
Thus, for $R=2(a-c_{\mathbf{K}})$, all consecutive Kepler ellipses pass through the second focus $F'$.

For this parameter value, the foci-circle and the foci-caustic circle are tangent to each other. Drawing the consecutive foci-secants $s_{i,i+1}$, $s_{i+1,i+2}$, and so on (compare Figure \ref{fig:focal_reflection}), we see that the second orbital foci, and thus the major axes of the Kepler ellipses themselves, tend to the x-axis, with the limiting Kepler ellipse degenerating into a line segment.

\end{proof}

\begin{figure}[htb]
\centering
\includegraphics[scale=1.1]{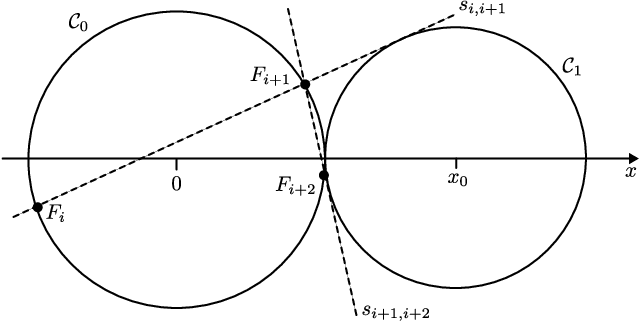}
\caption{Construction of consecutive second foci $F_i$ through secants between the two circles.}
\label{fig:focal_reflection}
\end{figure}

\section{Associated elliptic curve and Cayley-condition}
\subsection{The associated elliptic curve}
\label{Subsection: The associated elliptic curve}

As in \cite{JZ2}, with the help of the foci-circle and the foci-caustic circle, we reduce the dynamics of the Kepler billiards to a standard situation described by the Poncelet porism. We can now directly follow the classical work of Griffiths-Harris \cite{GH}. In the projective space $\mathbb{CP}^2$, the homogenized complexifications of these circles are given respectively by
\begin{align*}
{\cal{C}}_0&:~x^2+y^2-R^2z^2=0,\\
{\cal{C}}_1&:~x^2-2x_0xz+x_0^2z^2+y^2-r_0^2z^2=0.
\end{align*}

These curves are associated respectively to the symmetric matrices
\begin{align*}
Q_0[x,y,z]&\mapsto Q_0=\begin{pmatrix}
1 & 0 & 0\\
0 & 1 & 0\\
0 & 0 &-R^2
\end{pmatrix},\\
Q_1[x,y,z]&\mapsto Q_1=\begin{pmatrix}
1 & 0 & -x_0\\
0 & 1 & 0\\
-x_0 & 0 &x_0^2-r_0^2
\end{pmatrix}.
\end{align*}
The curves are singular if and only if their associated symmetric matrices are singular. This occurs precisely when $R=0$ for ${\cal{C}}_0$ or $r_0=0$ for ${\cal{C}}_1$. We exclude these two special cases from the further analysis and return to them later by considering special low-periodic orbits.
Their dual conics, ${\cal{C}}_0^*$ and ${\cal{C}}_1^*$, consisting of their respective tangent lines, are likewise non-singular in the dual projective space $\mathbb{CP}^{2*}$. Their associated symmetric matrices are $Q_0^{-1}$ and $Q_1^{-1}$, respectively.

We define the incidence variety

$$
E=\{(p,\xi)\in {\cal{C}}_1\times {\cal{C}}_0^*\mid p\in\xi\}.
$$

If ${\cal{C}}_0$ and ${\cal{C}}_1$ intersect transversely, then $E$ is a smooth algebraic curve of genus one in $\mathbb{CP}^2\times\mathbb{CP}^{2*}$. After choosing a suitable base point as the identity element, $E$ acquires the structure of an elliptic curve.

A convenient description of this elliptic curve, introduced in \cite{GH2}, also leads to Cayley's criterion for the existence of $n$-periodic trajectories. One considers the elliptic curve arising from the pencil of quadrics generated by ${\cal{C}}_0$ and ${\cal{C}}_1$. This curve is birationally equivalent to $E$ and is given by

$$
{\cal{D}}:~ y^2=-\det(t\cdot Q_0+Q_1),
$$
with the point at infinity chosen as the identity element.

The elliptic curve is then birationally equivalent to
$${\cal{D}}:y^2=-(t+1)\cdot (R^2t^2-(x_0^2-r_0^2-R^2)t+r_0^2).$$

The shift on this elliptic curve is given by integrating the invariant differential form from infinity to a point above 0. Since this leads to the same formula for the shift as already highlighted in \cite{JZ2}, we, at this point, refer to the former study.

Although the four possible choices of signs for $Q_0$ and $Q_1$ define equivalent complex pencils, the choice of signs is relevant when considering the real structure. Indeed, changing the overall sign of the determinant replaces the equation $y^2=F(t)$ by $y^2=-F(t)$, and the corresponding complex isomorphism is given by $y\mapsto iy$. This transformation is not defined over the real numbers and therefore does not preserve the standard real locus. Instead, it changes the antiholomorphic involution whose fixed-point set represents the real part of the incidence curve.
For the choice of signs used above, namely ${{\cal{D}}:~ y^2=-\det(t\cdot Q_0+Q_1)}$, the birational transformation constructed by Griffiths and Harris is compatible with complex conjugation. Hence it restricts to a birational transformation over the real numbers and identifies the real part of the incidence curve $E$ with the real locus ${\cal{D}}(\mathbb{R})$. For the other equivalent sign choices, the induced real structure is given by a different antiholomorphic involution, and its fixed-point locus does not in general coincide with the standard real locus of the corresponding affine equation. Thus, for the subsequent analysis of the real components, the chosen sign convention is essential.




\subsection{Topology of the real part of the associated elliptic curve} 
We now analyze the real parts of the elliptic curves on which the physical dynamics is linearized. 

For this the elliptic curve alternatively can be written in the form
$$\mathcal{D}: y^2=-R^2(t+1)(t+t_1)(t+t_2),$$
where
\begin{align*}
t_1&=\frac{R^2+r_0^2-x_0^2}{2R^2}+\frac{1}{2}\sqrt{\left(\frac{R^2+r_0^2-x_0^2}{R^2}\right)^2-4\frac{r_0^2}{R^2}},\\
t_2&=\frac{R^2+r_0^2-x_0^2}{2R^2}-\frac{1}{2}\sqrt{\left(\frac{R^2+r_0^2-x_0^2}{R^2}\right)^2-4\frac{r_0^2}{R^2}}.
\end{align*}
The sign of the normalized discriminant  of the quadratic equation
$$ (t+t_{1}) (t+t_{2})=0$$
is
$$\Delta:=(R^2+r_{0}^{2}-x_{0}^{2})^{2}- 4 r_{0}^{2} R^{2},$$
which determines the number of ramification points on the real part of $\mathcal{D}$. A short computation leads to the factorization
$$\Delta=(x_{0}+r_{0}-R) (x_{0}-r_{0}+R) (x_{0}-r_{0}-R)  (x_{0}+r_{0}+R). $$
The discriminant $\Delta$ can take three different signs depending on the relative distances of the foci-circle ${\cal{C}}_0$ and the foci-caustic circle ${\cal{C}}_1$:
\begin{enumerate}
\item $\Delta <0$ if ${\cal{C}}_0$ and ${\cal{C}}_1$ intersect in two points. In these cases two of the three finite ramification points of the elliptic curve ${\cal{D}}$ are nor real. The real part of ${\cal{D}}$ thus is topological equivalent to a circle.
\item $\Delta =0$ if ${\cal{C}}_0$ and ${\cal{C}}_1$ touch in one point or are identical. In these cases ${\cal{D}}$ is singular.
\item $\Delta >0$ if ${\cal{C}}_0$ and ${\cal{C}}_1$ do not intersect. In these cases the elliptic curve is regular and has the topology of the union of two circles.
\end{enumerate}

\subsection{Caley's condition on $n$-periodicity}
\label{sec:Cayley}
Explicit conditions for Poncelet $n$-gons were found by Cayley and reviewed in modern terms by Griffiths-Harris \cite{GH}. These conditions can be applied directly in our setting.

\vspace*{0,5cm}
\noindent
\textit{Expanding the function at $t=0$:
$$\sqrt{-\det(t\cdot Q_0+Q_1)}\dot{=}A_0+A_1\cdot t+A_2\cdot t^2+\dots.$$
For $n\geq 3$, the condition on $n-$periodic orbits is given by
\begin{align*}
\det\begin{pmatrix}
A_2 & A_3 & \dots & A_{m+1}\\
A_3 & A_4 & \dots & A_{m+2}\\
\vdots & & & \vdots\\
A_{m+1} & A_{m+2} & \dots & A_{2m}
\end{pmatrix}&=0~~~\mbox{for odd}~n=2m+1,\\
&\\
\det\begin{pmatrix}
A_3 & A_4 & \dots & A_{m+1}\\
A_4 & A_5 & \dots & A_{m+2}\\
\vdots & & & \vdots\\
A_{m+1} & A_{m+2} & \dots & A_{2m-1}
\end{pmatrix}&=0~~~\mbox{for even}~n=2m.
\end{align*}
}

The first few coefficients read:
\begin{align*}
A_0&=\sqrt{r_0^2},\\
A_1&=\frac{2r_0^2+R^2-x_0^2}{2\sqrt{r_0^2}},\\
A_2&=\frac{\sqrt{r_0^2}\cdot (4r_0^2R^2-R^4+2R^2x_0^2-x_0^4)}{8r_0^4},\\
&\vdots
\end{align*}

Similar to \cite{JZ2}, we can also determine $A_i$ iteratively. Let
$$C_{0}=r_{0}^{2}, \, C_{1}=( 2 r_{0}^{2}+R^{2}-x_{0}^{2}), \, C_{2}=2 R^{2}-x_{0}^{2}+r_{0}^{2}, \, C_{3}=R^{2},$$

then
\[
\begin{aligned}
& A_{0}^{2}=C_{0}, \\ 
& 2 A_{0} A_{1}=C_{1}, \\
& 2 A_{0} A_{2} + A_{1}^{2}=C_{2}, \\
& 2 A_{0} A_{3}+2 A_{1} A_{2}=C_{3}, \\
& 2 A_{0} A_{4}+2 A_{1} A_{3}+A_{2}^{2}=0, \\
& 2 A_{0} A_{5}+2 A_{1} A_{4}+2 A_{2} A_{3}=0, \\
& 2 A_{0} A_{6}+2 A_{1} A_{5}+2 A_{2} A_{4}+A_{3}^{2}=0, \\
& \cdots \cdots \cdots
\end{aligned}
\]


Note that Cayley's condition does not include 1- and 2-periodic orbits, for which a separate analysis is needed. In our case, these are most often confocal conic sections.

As an example, the conditions for period 3 and 4, as derived in \cite{JZ2, DR2}, read as follows. 

\textbf{$n=3$}: 
$$A_2=0 ~\Leftrightarrow ~(2Rr_0+x_0^2-R^2)\cdot (2Rr_0-x_0^2+R^2)=0.$$

\textbf{$n=4$}:
$$A_3=0~\Leftrightarrow ~ (R^2-r_0^2)\cdot (2r_0^2(R^2+x_0^2)-(R^2-x_0^2)^2)=0.$$

\subsection{Examples for low periodic orbits}
\textbf{Case} $n=2$:

Since the case of two-periodic orbits is not governed by Cayley's conditions, we need to analyze it in detail. There is an infinite family of two-periodic orbits corresponding to $R=0$ for both hyperbolic arc boundaries. In this case the Kepler orbits correspond to confocal ellipses with foci $F$ and $F'$. This is to be expected, due to the well-known fact that confocal ellipses and hyperbolas are orthogonal to each other.

\begin{figure}[!htb]
\centering
\includegraphics[scale=0.73]{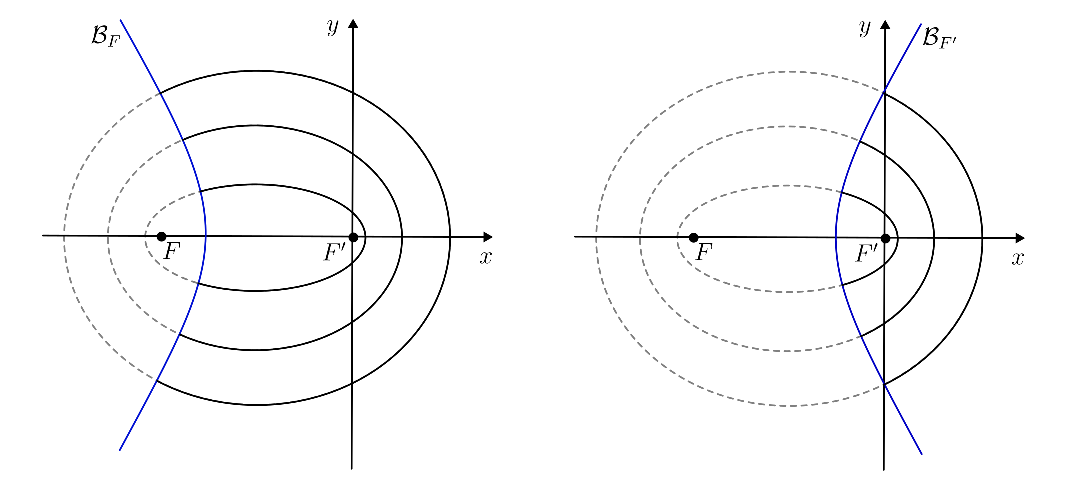}
\caption{Two periodic orbits for both boundaries corresponding to families of confocal ellipses.}
\label{fig:2_periodic}
\end{figure}

\noindent
\textbf{Case} $n=3$:

From Cayley's condition, we find that the condition for 3-periodicity \cite{DR2} is $A_2=0$, which reads

$$
(2Rr_0+x_0^2-R^2)\cdot (2Rr_0-x_0^2+R^2)=0,
$$
for both hyperbolic arc boundaries, respectively. For each admissible range, suitable configurations can be found. Figure \ref{fig:BF_period_3} displays two 3-periodic configurations for each hyperbolic arc, respectively.

\begin{figure}[!htb]
\centering
\includegraphics[scale=0.65]{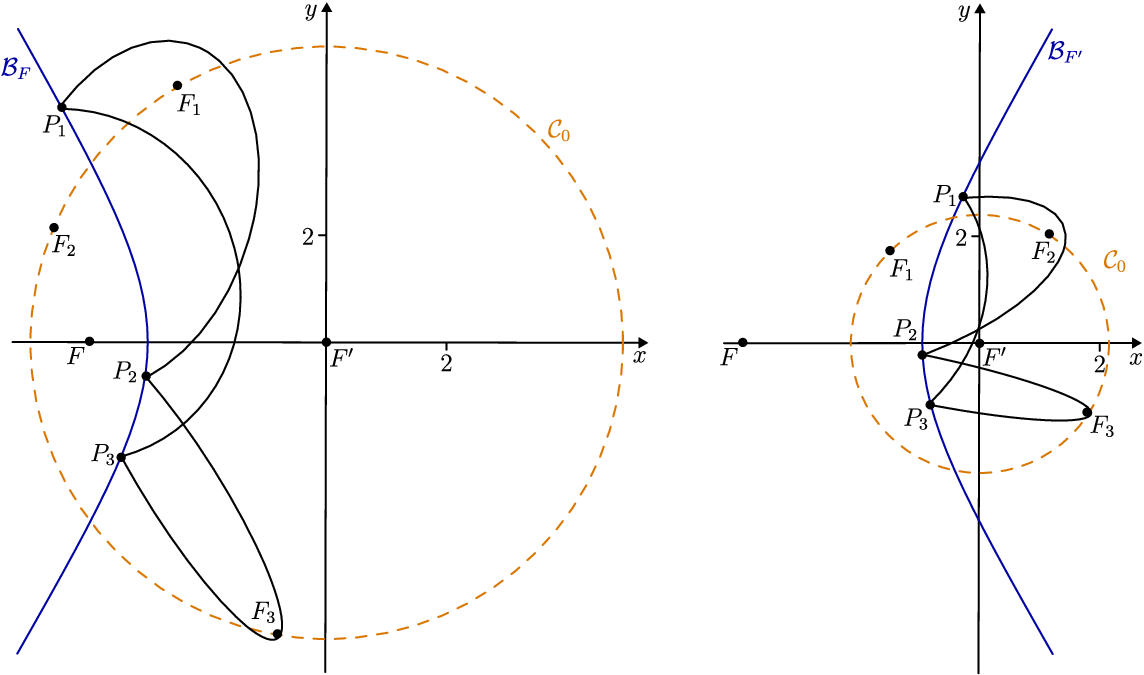}
\caption{Examples of 3-periodic orbits. \\
\textit{Left} for ${\cal{B}}_F$ and $(a_{\mathbf{K}},c_{\mathbf{K}},a,R)=(1,2,3,5)$. \\ \textit{Right} for ${\cal{B}}_{F'}$ and $(a_{\mathbf{K}},c_{\mathbf{K}},a,R)\approx (1,2,3,2.17)$.}
\label{fig:BF_period_3}
\end{figure}



\section{Limiting case 1: The parabolic boundary enclosing $F'$}
\label{sec:limits}
We may now recover two important cases previously studied, namely, that of a parabolic boundary with a constant gravitational force field parallel to its major axis \cite{Mas, J} as well as the classical integrable Boltzmann system e.g. \cite{Felder, GR}. 

We start our discussion by considering the limiting case obtained from the hyperbolic boundary arc enclosing $F'$.

\subsection{The limiting Procedure}

We consider the case ${\cal{B}}_{F'}$ of hyperbolic boundary enclosing $F'=(0,0)$, given by
$${\cal{B}}_{F'}:\frac{(x+c_{\mathbf{K}})^2}{a_{\mathbf{K}}^2}-\frac{y^2}{c_{\mathbf{K}}^2-a_{\mathbf{K}}^2}=1,~~ c_{\mathbf{K}}-a_{\mathbf{K}}>0,~~ x \geq a_{\mathbf{K}}-c_{\mathbf{K}}.$$
 The center of the Kepler force is located at $F=(-2c_{\mathbf{K}},0)$.

Defining 
\begin{align*}
    q_{\mathbf{K}}&=\frac{c_{\mathbf{K}}+a_{\mathbf{K}}}{2},\\
    p_{\mathbf{K}}&=\frac{c_{\mathbf{K}}-a_{\mathbf{K}}}{2},
\end{align*} 
we can rewrite the boundary equation as
\begin{align*}
-\frac{y^2}{4p_{\mathbf{K}}}+\frac{x^2+2x(q_{\mathbf{K}}+p_{\mathbf{K}})+(q_{\mathbf{K}}+p_{\mathbf{K}})^2-(q_{\mathbf{K}}-p_{\mathbf{K}})^2}{(q_{\mathbf{K}}-p_{\mathbf{K}})^2}\cdot q_{\mathbf{K}}&=0\\
-\frac{y^2}{4p_{\mathbf{K}}}+\frac{x^2+2x(q_{\mathbf{K}}+p_{\mathbf{K}})+4p_{\mathbf{K}}q_{\mathbf{K}}}{(q_{\mathbf{K}}-p_{\mathbf{K}})^2}\cdot q_{\mathbf{K}}&=0.
\end{align*}
For any $N>0$ with $|x|<N$, we carry out the limit $q_{\mathbf{K}}\rightarrow \infty$ with $p_{\mathbf{K}}=const.$. The boundary equation becomes
$$-\frac{y^2}{4p_{\mathbf{K}}}+2x+2p_{\mathbf{K}}=0~~\leftrightarrow ~~ x=\frac{y^2}{8p_{\mathbf{K}}}-2p_{\mathbf{K}},$$
i.e. the equation of a parabola with focus at $F'=(0,0)$. Since this is true for every $N \ge 0$, the hyperbolic boundaries tend to a parabola sharing a focus. 

In this limiting procedure, the foci-circle remains unaffected, but the Kepler ellipses focused at $F$ are stretched so much that they become parabolas themselves. 


The limiting system then corresponds to the integrable billiard model of constant gravitational force with parabolic boundary studied in \cite{Mas, J}.

The coordinate switch $(x,y)\rightarrow (y,x)$ together with setting 
$$f=2p_{\mathbf{K}}~~(=c_{\mathbf{K}}-a_{\mathbf{K}})$$ 
lead us to the more familiar case in which the directrix of the parabolic boundary is horizontal, thus parallel to the ground, and the constant force pushes the particle downward along the $y$-direction. The boundary is now written as
$${\cal{B}}_{P'}:y=\frac{x^2}{4f}-f,$$
where $f=2p_{\mathbf{K}}$ is the focal distance of the parabola with focus at $F'=(0,0)$. The foci-circle is again 
$${\cal{C}}_0:~x^2+y^2=R^2.$$

\subsection{The foci-caustic circle}

As a limiting case of families of hyperbolic arcs, it is no surprise that the foci-caustic circle also exists in this limiting system. We state this in the following theorem, and, in addition, make the mass of the moving particle and the gravitational constant explicit to establish a link to further applications in physics. Note that alternatively the associated foci-caustic circle may be obtained via a suitable limiting procedure. However, one needs to carefully define how the Kepler ellipses transform in this limit, i.e., the semi-major axis $a$ is a non-trivial function of $a_{\mathbf{K}},~c_{\mathbf{K}}$ and $R$. We therefore take the approach of directly considering parabolic arcs as trajectories and deducing the foci-caustic circle right from this setup.

\begin{theorem}
Let a particle of mass $m$ and with total energy $E$ bounces elastically at a parabolic boundary 
$${\cal{B}}_{P'}:y=\frac{x^2}{4f}-f$$ 
with focal length $f$ under the influence of a constant gravitational force $F=-mg$. Then the lines of consecutive  foci $F_i F_{i+1}$ of flight parabolas are tangent to a \textit{foci-caustic circle}
$${\cal{C}}_1:x^2+(y-y_0)^2=r_0^2,$$
where
$$y_0=\frac{\left(\frac{E}{mg}+2f\right)^2-R^2}{4f},$$
and
$$r_0^2=\left(\frac{R^2-(\frac{E}{mg})^2+4f^2}{4f}\right)^2.$$
\end{theorem} 

We include a direct proof.

\begin{proof}
Let $y=h_{max}=\frac{E}{mg}$ be the directrix of the flight parabolas with total energy $E$. By construction, the consecutive foci $F_i$ of the flight parabolas may be regarded as the intersection points of the foci-circle ${\cal{C}}_0$
and a second circle around the point of reflection $P$ (see \cite{J}) 

$${\cal{C}}_2:(x-x_P)^2+(y-y_P)^2=(h_{max}-y_P)^2,~~(x_P,y_P)\in{\cal{B}}_{P'}.$$

This is a direct consequence of all flight parabolas sharing the same directrix $h_{max}=y_F+2f$, where $y_F$ is the $y-$coordinate of the parabola focus, as well as conservation of energy at the point of reflection. For a geometric proof based on congruency consult \cite{Mas}.
Inserting the expression of ${\cal{C}}_0$ into ${\cal{C}}_2$, we obtain a new function

\begin{align*}
    F(x,y,x_P)&=R^2-h_{max}^2+x_P^2-2x\cdot x_P+2(h_{max}-y)\cdot y_P(x_P)\\
    &=R^2-h_{max}^2+x_P^2-2x\cdot x_P+2(h_{max}-y)\cdot \left(\frac{x_P^2}{4f}-f\right).
\end{align*}

To obtain the envelope curve, we need to solve the system of equations
\begin{align*}
    F(x,y,x_P)&=0,\\
    \frac{\partial F(x,y,x_P)}{\partial x_P}&=0.
\end{align*}

From the second equation
$$\frac{\partial F(x,y,x_P)}{\partial x_P}=0$$
we obtain 
$$x_P=\frac{2fx}{2f+h_{max}-y}.$$

Inserting this into the first equation and performing some algebraic manipulations, including completing the square, we find
\begin{align*}
R^2-h_{max}^2-2f(h_{max}-y)~~~~~~~~~~~~~~~~~~~~~~~~~~~~~~~~~~~~~~~~~~~~\\
+\frac{1}{(2f+h_{max}-y)^2}\cdot \underbrace{\left(4f^2x^2-4fx^2(2f+h_{max}-y)+2(h_{max}-y)fx^2\right)}_{=-2fx^2(h_{max}+2f-y)}=0\\
\leftrightarrow ~ \frac{h_{max}^2-R^2}{2f}\cdot (2f+h_{max}-y)+(h_{max}-y)(2f+h_{max}-y)+x^2=0\\
\leftrightarrow ~x^2+\left(y-\frac{(h_{max}+2f)^2-R^2}{4f}\right)^2=R^2-h_{max}^2+\left(\frac{R^2-h_{max}^2-4f^2}{4f}\right)^2,
\end{align*}

which represents the foci-caustic circle
$${\cal{C}}_1:x^2+(y-y_0)^2=r_0^2,$$
with
\begin{align*}
    y_0&=\frac{(h_{max}+2f)^2-R^2}{4f},\\
    r_0^2&=R^2-h_{max}^2+\left(\frac{R^2-h_{max}^2-4f^2}{4f}\right)^2=\left(\frac{R^2-h_{max}^2+4f^2}{4f}\right)^2.
\end{align*}
\end{proof}

\subsection{Admissible range of parameters}

As in \cite{J}, the focus $F_\varphi=(F_x,F_y)=(R\cos(\varphi),R\sin(\varphi)),~ \varphi\in [0,2\pi[$
and directrix $h_{max}=\frac{E}{mg}$ determines the flight parabola by
$$y=-\frac{(x-R\cos(\varphi))^2}{2(h_{max}-R\sin(\varphi))}+\frac{h_{max}+R\sin(\varphi)}{2}.$$

When the particle hits the boundary, the flight parabolas intersect the boundary at two points each (counted with multiplicities). The discriminant of the associated quadratic equation
$$\frac{x^2}{4f}-f=-\frac{(x-R\cos(\varphi))^2}{2(h_{max}-R\sin(\varphi))}+\frac{h_{max}+R\sin(\varphi)}{2},$$
has to be positive. This leads to the inequality 
$$(h_{max}+2f)^2>R^2.$$ 

Furthermore, we have $F_y<h_{max}$, from which it follows that
$$R\sin(\varphi)<h_{max} \Leftrightarrow R<h_{max}~\mbox{for}~\varphi=\frac{\pi}{2}.$$

We can summarize the findings as follows. For fixed $f$, the set of admissible parameters $(E,R)$ is given by:
\begin{align*}
    0\leq R &<\frac{E}{mg}+2f,~~~\mbox{for}~E>0,\\
    -\frac{E}{mg}<R&<\frac{E}{mg}+2f,~~~\mbox{for}~-mgf<E\leq 0.
\end{align*}

See Figure \ref{fig:admiss_range_parabolic} for a graphical representation. 

Note that in the extreme case $R=\frac{E}{mg}+2f$ the only admissible point of reflection with the boundary is $x=0$. In this case, the foci-circle coincides with the foci-caustic circle.

\begin{figure}[!htb]
\centering
\includegraphics[scale=0.8]{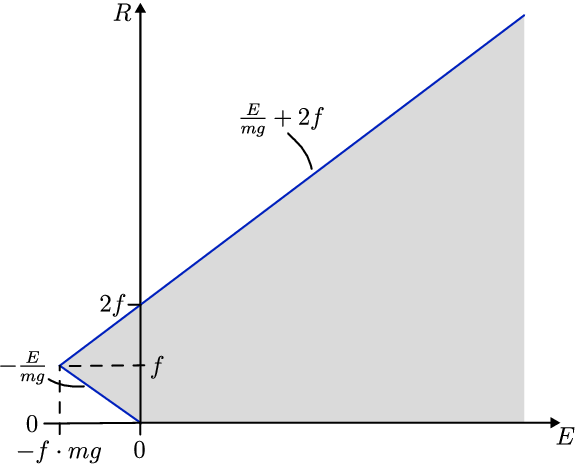}
\caption{Admissible range $(E,R)$ for physical reflections of the flight parabola with the boundary ${\cal{B}}_{P'}$.}
\label{fig:admiss_range_parabolic}
\end{figure}

\subsection{Relative positions of the two circles} 
As in Section \ref{subsec:BF circles}, we analyze the relative positions for the two circles ${\cal{C}}_0$ and ${\cal{C}}_1$.  We consider a fixed value of $h_{max}=\frac{E}{mg}$ and vary $R$ according to the admissible range from low to high values.

When considering $r_0^2$, we see that $r_0=0$ for $R=\sqrt{h_{max}^2-4f^2}$. Since we interpret $r_0$ as a positive length, we thus obtain two cases to consider
$$r_0=\begin{cases} \frac{h_{max}^2-4f^2-R^2}{4f}~~\mbox{for}~0\leq R\leq \sqrt{h_{max}^2-4f^2}~~\mbox{domain (I)}\\
\frac{R^2+4f^2-h_{max}^2}{4f}~~\mbox{for}~\sqrt{h_{max}^2-4f^2}<R\leq h_{max}+2f ~~\mbox{domain (II)}
\end{cases}
$$
We consider different cases related to the two different choices of $r_0$.

\vspace*{0.2cm}
\noindent
\textbf{Case 1} $h_{max}>2f$:
\vspace*{0.2cm}

Define $d_1=y_0-r_0-R$ as the minimum distance between the two distinct circles. Analogously to the analysis in Section \ref{subsec:BF circles} we determine the zeros of $d$ regarded as a function of $R$ depending on the representation of $r_0$ for the two different domains $(I)$ and $(II)$.

For the first domain $0<R<\sqrt{h_{max}^2-4f^2}$ we have
$$d_1=\frac{h_{max}^2+4h_{max}f+4f^2-R^2}{4f}-\frac{h_{max}^2-4f^2-R^2}{4f}-R=0.$$
The only positive root is given by $R=h_{max}+2f$ which is not part of the first domain; thus, the foci-caustic circle ${\cal{C}}_1$ lies outside the foci-circle ${\cal{C}}_0$. 

At $R=\sqrt{h_{\max}^2-4f^2}$, the circle ${\cal{C}}_1$ shrinks to a point, which corresponds to 2-periodic orbits.

For $\sqrt{h_{\max}^2-4f^2}<R<h_{max}$ we can perform the same analysis for determining the roots of $d$, but this time with the expression for $r_0$ of the second domain
$$d_1=\frac{h_{max}^2+4h_{max}f+4f^2-R^2}{4f}-\frac{R^2-h_{max}^2+4f^2}{4f}-R=0.$$
A direct calculation shows that the only root lying in the suitable domain is given by $R=h_{max}$. Thus for $\sqrt{h_{max}2-4f^2}<R<h_{max}$ the two circles are distinct but $r_0$ increases. For $R=h_{max}$ the circles are tangent to another. Further increasing $R$ the two circles intersect until finally they are identical for $R=h_{max}+2f$.


\vspace*{0.2cm}
\noindent
\textbf{Case 2} $h_{max}=2f$:
\vspace*{0.2cm}

We can use the results of the calculations carried out in Case 1:
When $R=0$, both circles degenerate into points. For $0<R<h_{max}$ we find two distinct circles. The two circles are tangent when $R=h_{max}$. For $h_{max}<R<h_{max}+2f$, the circles intersect transversally at two points. Finally, when $R=h_{max}+2f$, the two circles coincide.

\vspace*{0.2cm}
\noindent
\textbf{Case 3} $0<h_{max}<2f$:
\vspace*{0.2cm}

Again, using the results of Case 1: For $R=0$ we have $r_0>0$, thus we have two distinct circles with ${\cal{C}}_0$ degenerated to a point. When $0<R<h_{max}$, the two circles are disjoint, and are tangent to each other when $R=h_{max}$. For $h_{max}<R <h_{max}+2f$ the two circles intersect transversally at two points. Finally, when $R=h_{max}+2f$, the two circles coincide.

\vspace*{0.2cm}
\noindent
\textbf{Case 4} $h_{max}=0$:
\vspace*{0.2cm}

Again, using limits of former results: For $R=0$ we have $y_0=r_0>0$. Thus  ${\cal{C}}_0$ degenerates to a point with ${\cal{C}}_0\in {\cal{C}}_1$.
For $0<R<h_{max}+2f$ the two circles intersect transversally at two points.  They are identical when $R=h_{max}+2f$.

\vspace*{0.2cm}
\noindent
\textbf{Case 5} $-f<h_{max}<0$: 
\vspace*{0.2cm}

Again, using limits of the previous results: Now the roles of ${\cal{C}}_0$ and ${\cal{C}}_1$ are interchanged; thus, we define a new minimal distance of the two circles by $d_2=y_0+R-r_0$. Again, if $d_2>0$, the two circles are disjoint, if $d_2=0$, they are tangent, and if $d_2<0$, ${\cal{C}}_0$ lies inside ${\cal{C}}_1$. By an analogous calculation for the zeros of $d_2$ we find (with respect to the admissible range):
When $|h_{max}|< R < h_{max}+2f$, the two circles intersect transversally at two points. 
When $R=h_{max}+2f$ the two circles coincide again.

\vspace*{0.2cm}
\noindent
\textbf{Case 6} $-f=h_{max}$: 
Using limiting results from Case 5 we find that $R=-h_{max}=f$. Thus the two circles are identical.
\vspace*{0.2cm}

\vspace*{0,5cm}
\textbf{Note:} The focal reflection property carries over (see subsection \ref{subsec:focal_reflection_property}) to the parabolic boundary with constant gravitational force namely, for all cases, when $R=h_{max}$, all associated flight parabolas pass through the focus $F'=(0,0)$ of ${\cal{B}}_{P'}$. Furthermore, consecutive flight parabolas asymptotically converge to the vertical line (remember the change of coordinate system) $x=0$.

\subsection{$n$-periodic orbits for small $n$}
From the analysis of the relative positions of the foci- and foci-caustic circles, we see that in this parabolic case, for a billiard with constant gravitational force, we also have the Poncelet property. 

An interesting case occurs for $n=2$. Indeed, every confocal flight parabola corresponding to $R=0$ yields a two-periodic orbit. In addition, there are other two-periodic orbits when $\frac{E}{mg}\geq 2f$. They correspond to the case $r_0=0$, so the parameter values satisfy $R=\sqrt{\left(\frac{E}{mg}\right)^2-4f^2}$. See Figure \ref{fig:period_2_parabolic}.

\begin{figure}[!htb]
\centering
\includegraphics[scale=0.7]{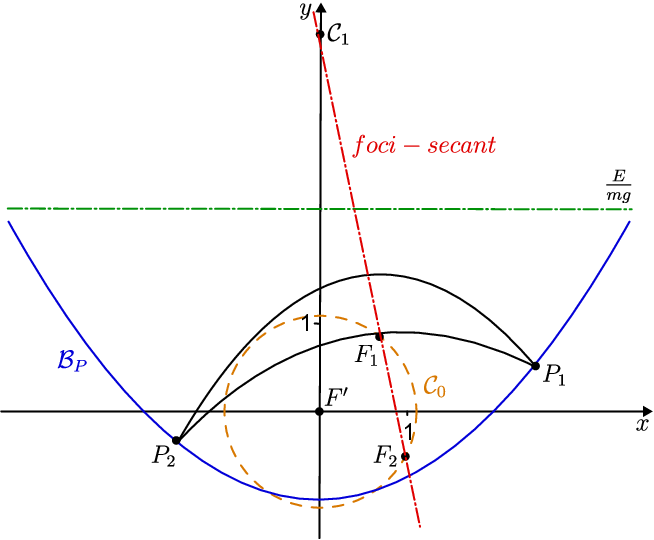}
\caption{Non-trivial 2-periodic orbit for $f=1$, $h_{max}\approx 2.3$ and $R\approx 1.1$.}
\label{fig:period_2_parabolic}
\end{figure}

We also illustrate a further example of  $4$-periodicity. For $4-$periodic orbits, the corresponding condition is
$$A_3=0~\Leftrightarrow ~ (R^2-r_0^2)\cdot (2r_0^2(R^2+y_0^2)-(R^2-y_0^2)^2)=0,$$

Figure \ref{fig:period_4_parabolic} displays some $4$-periodic orbits with $f=1,~\frac{E}{mg}=3$ and {${R\approx 3.385}$}. Note that the family of $4-$periodic orbits appear symmetrically with respect to the $y-$axis.

\begin{figure}[!htb]
\centering
\includegraphics[scale=0.7]{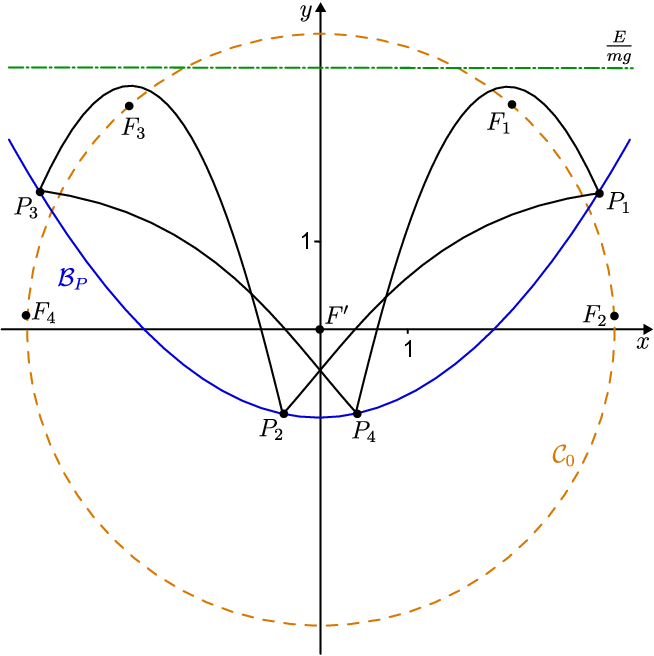}
\caption{Symmetry $4-$periodic orbit.}
\label{fig:period_4_parabolic}
\end{figure}

\section{Limiting Case 2: The parabolic boundary enclosing $F$}
\label{sec:limits2}
\subsection{Another parabolic boundary case}
We now consider ${\cal{B}}_{F}$ in a system of coordinates such that $F$ is at the origin: 
$${\cal{B}}_F:\frac{(x-c_{\mathbf{K}})^2}{a_{\mathbf{K}}^2}-\frac{y^2}{c_{\mathbf{K}}^2-a_{\mathbf{K}}^2}=1,\qquad x\leq c_{\mathbf{K}}-a_{\mathbf{K}}.$$

Defining again
\begin{align*}
p_{\mathbf{K}}:=\frac{c_{\mathbf{K}}-a_{\mathbf{K}}}{2},\\
q_{\mathbf{K}}:=\frac{c_{\mathbf{K}}+a_{\mathbf{K}}}{2},
\end{align*}
and taking this time the limit $p_{\mathbf{K}}\rightarrow \infty$ with $q_{\mathbf{K}}=const.$ along a suitable family of these hyperbolic branches, we obtain as limit the parabola focused at $F$
$${\cal{B}}_P:x=-\frac{y^2}{4f}+f,$$
which has focal length $f=2q_{\mathbf{K}}~~(=c_{\mathbf{K}}+a_{\mathbf{K}})$.

Along this limiting procedure we have $F'\rightarrow \infty$. As the semi-major axis $a$ of the Kepler orbits is assumed to be fixed in this limit, we need to send $R$ to $\infty$ as well. With an appropriate normalization, the foci-circles converge in the limit to a line of the form 
$${\tilde{\cal{C}}}_0:x=s_{0}.$$
We call this the \emph{foci-line}, which has been previously discussed in \cite{JZ2}.

\subsection{Admissible range for the parabolic boundary}
We determine the conditions for the admissible parameter range $(a,s_0)$ such that the corresponding Kepler flight ellipses with focus at $F$ can intersect the boundary ${\cal{B}}_P$.

Clearly, if $s_0\geq f$ we need to have $2a \ge s_0$.  For the case $s_0<f$, we consider the following proposition.
\begin{prop}
Let $s_0<f$. Then the family of Kepler ellipses with semi-major axis $a=f-\frac{s_0}{2}$ and second focus $F_i=(s_0,F_y)$ with $F_y\in[-\sqrt{4f(f-s_0)},\sqrt{4f(f-s_0)}]$ are tangent to the confocal parabola ${x=-\frac{y^2}{4f}+f}$ with focus $F=0$, which is the center of the attractive Kepler force.
\end{prop}
\begin{proof}
Let $s_0<f$ and choose an arbitrary position of the second Kepler focus $F_i$ along the foci-line $x=s_0$ such that the $y$-component $F_y$ of the second focus lies within the parabolic boundary $x=-\frac{y^2}{4f}+f$ which gives 
$$F_y\in[-\sqrt{4f(f-s_0)},\sqrt{4f(f-s_0)}].$$ 
We assume the Kepler ellipse to be tangent to the parabolic boundary (see Figure \ref{fig:family_tangency}).

\begin{figure}[htb]
\centering
\includegraphics[scale=1]{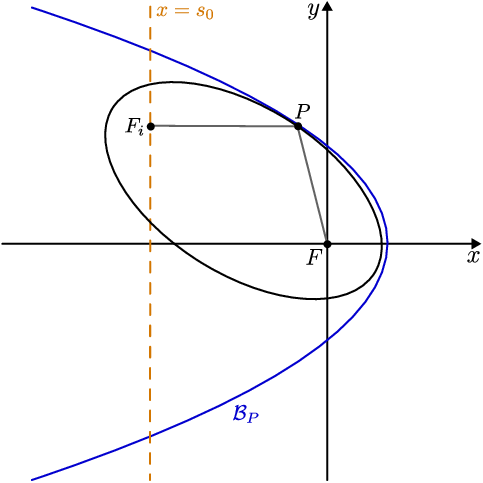}
\caption{Kepler ellipse with point of reflection tangent to ${\cal{B}}_P$ and foci-line $s_0<f$.}
\label{fig:family_tangency}
\end{figure}

Let $P=(x_P,y_P)$ be the point of reflection, i.e. by assumption the point of tangency of the Kepler ellipse with the parabolic boundary. Since the focus of the parabolic boundary is at the Kepler center, we may regard $\overline{FP}$ and $\overline{F_iP}$ as light rays. By the optical focusing property of the parabola we immediately know that $\overline{F_iP}$ must be parallel to the $x$-axis, thus $s_0<x_P$. For the semi-major axis of the tangent Kepler ellipse, we thus find, with ${x_P=-\frac{y_P^2}{4f}+f}$, that
\begin{align*}
2a&=|F_iP|+|PF|=x_P-s_0+\sqrt{x_P^2+y_P^2}\\
&=-\frac{y_P^2}{4f}+f-s_0+\sqrt{\left(-\frac{y_P^2}{4f}+f\right)^2+y_P^2}\\
&=-\frac{y_P^2}{4f}+f-s_0+\sqrt{\left(\frac{y_P^2}{4f}+f\right)^2}\\
&=2f-s_0=const.
\end{align*}
We thus conclude that there exists a family of Kepler ellipses tangent to the parabolic boundary with equal semi-major axis $a=f-\frac{s_0}{2}$ and second focus $F_i=(s_0,F_y)$, with $F_y$ restricted by the parabolic boundary as stated above.
\end{proof}
As a direct consequence of the above proposition, we have that the Kepler ellipses with $s_0<f$ undergo physical reflections at the parabolic boundary if $a>f-\frac{s_0}{2}$. We summarize the results in the following proposition.

\begin{prop}
The parameters $(a,s_0)$ for physical reflections of the Kepler ellipses with the boundary ${\cal{B}}_P$ are given by:
\begin{align*}
s_0&<2a,~~\mbox{if}~f\le s_0,\\
2f-2a&<s_0,~~\mbox{if}~ s_0< f.
\end{align*}
\end{prop}
\label{prop:admiss_range_B_P}

See Figure \ref{fig:admiss_range_parabolic_F} for a graphical representation.

\begin{figure}[htb]
\centering
\includegraphics[scale=0.8]{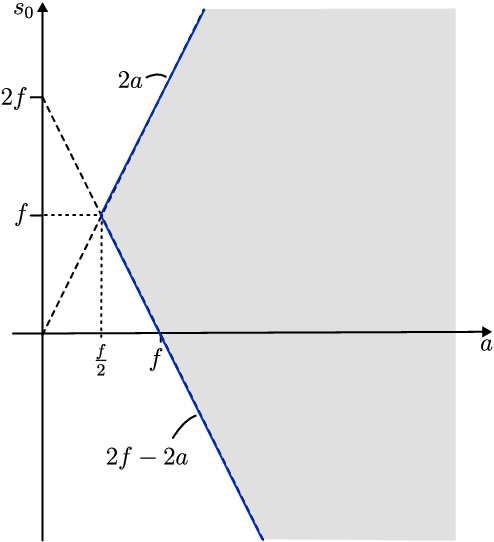}
\caption{Admissible domain for physical reflections along ${\cal{B}}_P$.}
\label{fig:admiss_range_parabolic_F}
\end{figure}

\subsection{The elliptic curve and periodic orbits}
Here we derive the associated elliptic curve for this limiting case via a limiting procedure.

\begin{theorem}
The elliptic curve associated to the parabolic boundary enclosing $F$ with focal length $f$ and foci-line $x=s_0$ is given by
$$\bar{{\cal{D}}}:\bar{y}^2=-(t+1)\cdot \left(2t+\frac{2a+s_0}{f}\right)^2.$$
\end{theorem}

\begin{proof}
We consider ${\cal{C}}_{0}$ and ${\cal{C}}_{1}$ in a family of systems with hyperbolic boundary enclosing $F$, shifting the origin to $F$. The expressions for the associated foci-circles and foci-caustic circles are respectively
\begin{align*}
{\cal{C}}_0&:(x-2c_{\mathbf{K}})^2+y^2=(2 c_{\mathbf{K}}-s_0)^2,\\
{\cal{C}}_1&:(x-x_0-2c_{\mathbf{K}})^2+y^2=r_0^2,
\end{align*}
where, again, we define the radius of the foci-circle as $R=2c_{\mathbf{K}}-s_0$ where $s_0$ is the signed distance from the foci-line to the parallel $y$-axis. Note that the expressions for $x_0$ and $r_0$ are given in Section \ref{sec:foci_caustic}.

As in Section \ref{sec:Cayley}, we consider the two associate quadratic matrices
$$Q_0[x,y,z]\mapsto Q_0=\begin{pmatrix}
1 & 0 &-2c_{\mathbf{K}}\\
0 & 1 & 0\\
-2c_{\mathbf{K}} & 0 & 4c_{\mathbf{K}}s_0-s_0^2
\end{pmatrix}$$
and 
$$Q_1[x,y,z]\mapsto Q_1=\begin{pmatrix}
1 & 0 &-2c_{\mathbf{K}}-x_0\\
0 & 1 & 0\\
-2c_{\mathbf{K}}-x_0 & 0 & (x_0+2c_{\mathbf{K}})^2-r_0^2
\end{pmatrix}$$
in projective space, to which the elliptic curve
\begin{align*}
y^2&=-\det(t\cdot Q_0+Q_1)\\
&=-(t+1)\cdot \left[(t+1)\cdot (4c_{\mathbf{K}}s_0-ts_0^2+(x_0+2c_{\mathbf{K}})^2-r_0^2)-(2c_{\mathbf{K}}t+2c_{\mathbf{K}}+x_0)^2\right].
\end{align*}
is associated.

By the corresponding limiting procedure we obtain the associated limiting elliptic curve as
$$\bar{y}^2=-(t+1)\cdot \left(2t+\frac{2a+s_0}{f}\right)^2.$$

\end{proof}

Periodic orbits with period $n \ge 3$ can be obtained again by Cayley's criterium as recalled in Section \ref{sec:Cayley}. 

The first two of such conditions read:
\begin{align*}
n=3:~A_2=0~~&\Leftrightarrow ~~ 2a+s_0=8f,\\
n=4:~A_3=0~~&\Leftrightarrow ~~ 2a+s_0=4f.
\end{align*}

Graphical representations of these two cases are shown in Figure \ref{fig:parabolic_F_orbits}.

\begin{figure}[!htb]
\centering
\includegraphics[scale=0.8]{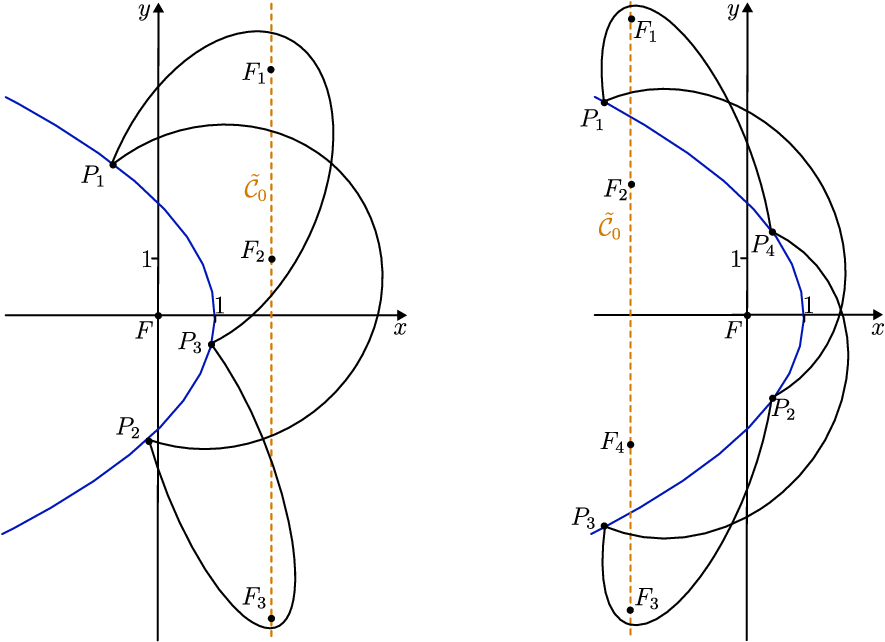}
\caption{\textit{Left:} Picture for a $3-$periodic orbit along ${\cal{B}}_P$ for $(f,a,s_0)=(1,3,2)$. \\
\textit{Right:} Picture for a $4-$periodic orbit along ${\cal{B}}_P$ for $(f,a,s_0)=(1,3,-2)$.}
\label{fig:parabolic_F_orbits}
\end{figure}

One difference as compared to many other cases is the following 

\begin{prop}
There exist no $2-$periodic orbits for the system with parabolic boundary enclosing $F$.
\end{prop}

\begin{proof}

For two-periodic orbits, either $R=0$ or $r_0=0$. The case $R=0$ cannot occur because the limiting procedure is equivalent to $R\rightarrow \infty$. The second case cannot occur since the foci-line is approximated by the foci-caustic circle in the large-$R$ limit, and thus $r_0\neq 0$. Hence, there exist no two-periodic orbits in this setup.
\end{proof}

\section{Limiting Case 3: The integrable Boltzmann system}
\label{sec:limits3}
The system, in which a particle under the influence of an (attractive) Kepler force is reflected at a line boundary, was first considered by Ludwig Boltzmann \cite{B} to illustrate his ergodic hypothesis. 
In \cite{GJ} this system was shown to be integrable where the conserved quantity corresponds to a foci-circle whose center is the mirror image of the Kepler force center with respect to the straight line. This integrable Boltzmann system was further analyzed in \cite{Felder} where an associated elliptic curve was constructed for the dynamics by considering involutions of consecutive Kepler orbits.  Periodicity conditions for this system were subsequently analyzed in \cite{GR}.
 
The integrable Boltzmann system can be seen as another limiting case in which $a_{\mathbf{K}}\rightarrow 0$. In this case both hyperbolic arcs ${\cal{B}}_F$ and ${\cal{B}}_{F'}$ degenerate into the straight line 
$$\ell: x_\ell=-c_{\mathbf{K}}.$$ 

The admissible parameter range can be carried over directly from the hyperbolic boundary cases in the limit $a_{\mathbf{K}}\rightarrow 0$. We get
\begin{align*}
R\geq 2c_{\mathbf{K}}-2a~~&\mbox{for}~0\leq R<c_{\mathbf{K}},\\
R\leq 2a~~&\mbox{for}~R\geq c_{\mathbf{K}}.
\end{align*}
For a graphical representation see Figure \ref{fig:admissible_range_boltzmann}.

\begin{figure}[htb]
\centering
\includegraphics[scale=0.9]{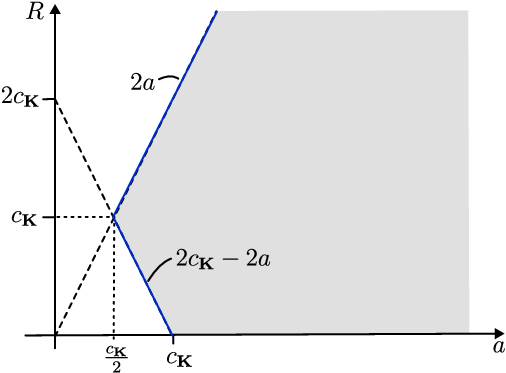}
\caption{Admissible range $(a,R)$ for the Boltzmann-system.}
\label{fig:admissible_range_boltzmann}
\end{figure}


The foci-caustic circle $S(x_{0}, r_{0})$ exists in this case. The functions $x_0$ and $r_0$ are now simply given by:
\begin{align*}
x_0&=-\frac{1}{2c_{\mathbf{K}}}\cdot (R^2-4a^2),\\
r_0&=2a.
\end{align*}
This directly follows by using the expressions for $x_0$ and $r_0$ in the hyperbolic boundary cases and taking the straightforward limit $a_{\mathbf{K}}\rightarrow 0$.
Note that in this system, another degeneracy appears: the radius of the foci-caustic circle is independent of $R$. The limiting procedure can also be applied directly to the associated elliptic curves for the hyperbolic boundaries which gives:


$$
y^2=-r_0^2\cdot (t+1)\cdot \left(t^2\cdot \frac{R^2}{r_0^2}-t\cdot \frac{x_0^2-r_0^2-R^2}{r_0^2}+1\right).
$$

By normalization, we consider the case $c_{\mathbf{K}}=1$ in order to facilitate comparison with previous studies, e.g. \cite{GJ,Felder}.  The total Kepler energy in an appropriate normalization is related to the orbital semi-major axis via  $E=-\frac{1}{2a}$. We define

$$
\tilde{R}:=R\cdot |E|,
$$
and

$$
D:=\frac{R^2E^2-4E^2-1}{2E}.
$$

The substitution
$$
t=-\frac{2(D+2E)}{\tilde{R}}\cdot \xi, ~~~y=\sqrt{\tilde{R}/r_0^2} \cdot y
$$
then transform the elliptic curve into
$$
y^2=\left[2(D+2E)\cdot \xi-\tilde{R}\right]\cdot \left[4\tilde{R}(D+2E)^2\xi^2+2(D+2E)(D^2+2DE-2)\xi +\tilde{R} \right].
$$

This is exactly the elliptic curve that G. Felder derived \cite{Felder}.
We thus establish a link between our study and his work. For the dynamical analysis, we have nothing to add to his work.


\section{Conclusive remarks}
In this work, we complete the analysis, started in \cite{JZ,JZ2}, of the (periodic) dynamics of a particle under the influence of an (attractive) Kepler potential with a conic section boundary. Analogously to \cite{JZ2}, we showed that when the boundary is one of the two possible hyperbolic arcs ${\cal{B}}_F$ or ${\cal{B}}_{F'}$ the consecutive second foci $F_i$ of the Kepler ellipses all lie on a common \textit{foci-circle} and that the secants formed by $\overline{F_iF_{i+1}}$ are tangent to another circle, the so-called \textit{foci-caustic circle}. Applying the results of the classical works on the relation between the Poncelet porism and elliptic curves \cite{GH, GH2} we determine the associated elliptic curve for each individual case from which conditions for $n$-periodic orbits can be derived.
By several limiting procedures, we discuss the dynamics of the related systems with parabolic arcs or a straight-line boundary (also known as \textit{integrable Boltzmann system}). In particular, we were able to reproduce results already obtained in \cite{Felder,GR} for the integrable Boltzmann system. 

In the plane, the conformal mapping $\mathbb{C} \to \mathbb{C}, z\mapsto z^2$  induces a correspondence between Kepler billiards and Hooke billiards \cite{Panov, TZ1}. The results for the dynamics with an elliptical boundary \cite{JZ2} as well as the results obtained in this work can be carried over to the Hooke billiards via this conformal correspondence.

\section*{Dedication and Acknowledgement}

\begin{center}
\textit{To Kai, celebrating his third 20th birthday - with admiration, and confidence that the sequence continues.}
\end{center}
 
L.Z. is supported by DFG ZH 605-4/1, Research Funds for Central Universities of China and the Liaoning Revitalization Talents Program.

\bibliographystyle{cas-model2-names} 

\end{document}